\documentclass[11pt,reqno]{amsart}
\usepackage{color}
\usepackage[english]{babel}
\usepackage{amsfonts}
\usepackage{amsmath}
\usepackage{amsthm}
\usepackage{amssymb}
\usepackage[misc]{ifsym}
\usepackage{bbm}
\usepackage{mathrsfs}
\usepackage{esint}
\usepackage{faktor}
\usepackage[utf8]{inputenc}
\usepackage{afterpage}
\usepackage[left=2.9cm,right=2.9cm,top=2.8cm,bottom=2.8cm]{geometry}
\usepackage{graphicx}
\usepackage{enumitem}
\usepackage[dvipsnames]{xcolor}
\usepackage[colorlinks=true,urlcolor=blue,citecolor=blue,linkcolor=blue,linktocpage,pdfpagelabels,bookmarksnumbered,bookmarksopen]{hyperref}

\newcommand{\N}{\mathbb{N}}

\newcommand{\R}{\mathbb{R}}
\newcommand{\A}{\mathcal{A}}
\newcommand{\D}{\mathcal{D}}
\newcommand{\F}{\mathcal{F}}
\newcommand{\G}{\mathcal{G}}
\newcommand{\J}{\mathcal{J}}
\newcommand{\K}{\mathcal{K}}
\newcommand{\M}{\mathcal{M}}

\newcommand{\dx}{\, {\rm d} x}

\newcommand{\dt}{\, {\rm d} t}

\newcommand{\dtau}{\, {\rm d} \tau}

\newcommand{\eps}{\varepsilon}
\newcommand{\loc}{{\rm loc}}
\newcommand{\Gr}{{\rm Gr}}

\newtheorem{lemma}{Lemma}
\newtheorem{thm}[lemma]{Theorem}
\newtheorem{prop}[lemma]{Proposition}

\theoremstyle{definition}
\newtheorem{defi}[lemma]{Definition}
\newtheorem{rmk}[lemma]{Remark}

\DeclareMathOperator*{\esssup}{ess \, sup}
\DeclareMathOperator*{\essinf}{ess \, inf}
\DeclareMathOperator*{\supp}{supp}
\DeclareMathOperator*{\diam}{diam}
\DeclareMathOperator*{\dom}{dom}

\begin{document}
\title[The obstacle problem for singular quasi-linear elliptic equations]{The obstacle problem for \\ singular quasi-linear elliptic equations}

\author[A. Barbagallo]{Annamaria Barbagallo}
\address[A. Barbagallo]{Dipartimento di Matematica e Applicazioni ``Renato Caccioppoli'', Università di Napoli Federico II, Via Cinthia, 80126 Napoli, Italy}
\email{annamaria.barbagallo@unina.it}

\author[U. Guarnotta]{Umberto Guarnotta}
\address[U. Guarnotta]{Dipartimento di Matematica e Informatica, Università degli Studi di Catania, Viale A. Doria 6,
95125 Catania, Italy}
\email{umberto.guarnotta@unict.it}

\begin{abstract}
Existence of solutions to an obstacle $p$-Laplacian problem exhibiting a singular, discontinuous reaction is proved. The reaction term may be discontinuous in a Lebesgue-negligible set. Moreover, solutions are shown to be locally $C^{1,\alpha}$ far away from the contact set. Under a differentiability hypothesis on the obstacle, solutions belong to $C^{1,\alpha}(\overline{\Omega})$.
\end{abstract}

\maketitle

{
\let\thefootnote\relax
\footnote{\Letter \quad Corresponding author: Annamaria Barbagallo (annamaria.barbagallo@unina.it).}
}
\setcounter{footnote}{0}

\section{Introduction}
Let $\Omega\subseteq \R^N$, $N\geq 2$, be a bounded domain of class $C^2$. We consider a Borel-measurable function $f:(0,+\infty)\to[0,+\infty)$ (called {\it reaction}) and a Lebesgue-measurable function $\phi:\Omega\to\R$ (called {\it obstacle}). The functions $f$ and $\phi$ satisfy, respectively, the hypotheses \hyperlink{Hf}{$({\rm H}_f)$}--\ref{Hphi} below (see Section 2). Given $p\in(1,\infty)$, we consider the set
\begin{equation}
\label{Kdef}
K := \{u\in W^{1,p}_0(\Omega): \, u\geq\phi \ \mbox{a.e. in} \ \Omega\}.
\end{equation}
In addition, we define
\begin{equation}
\label{subsupf}
\underline{f}(s):=\lim_{\delta\to 0^+}\essinf_{|t-s|<\delta} f(t)\quad\mbox{and}\quad
\overline{f}(s):=\lim_{\delta\to 0^+}\esssup_{|t-s|<\delta} f(t)\quad \mbox{for all} \ s\in(0,+\infty).
\end{equation}

The paper is devoted to existence of positive (a.e.\,in $\Omega$) solutions $u\in K$ to the inequality
\begin{equation}
\label{prob}
\tag{P}
\int_\Omega |\nabla u|^{p-2}\nabla u \nabla(v-u) \dx \geq \int_\Omega \eta(v-u) \dx \quad \mbox{for all} \ v\in K,
\end{equation}
where $\eta\in W^{-1,p'}(\Omega)\cap L^1_\loc(\Omega)$ is a suitable function satisfying
$$ \underline{f}(u(x)) \leq \eta(x) \leq \overline{f}(u(x)) \quad \mbox{for a.a.} \ x\in\Omega. $$
In addition, $C^{1,\alpha}$ regularity results for solutions to \eqref{prob} will be proved: see Theorem \ref{mainthm} below for a precise statement.

To the best of our knowledge, this is the first paper concerning existence of solutions to obstacle problems with singular reaction terms (i.e., $f(s)\to+\infty$ as $s\to 0^+$ is allowed). In addition, the reaction term considered here may be also discontinuous, leading to a singular variational-hemivariational problem.

The singular behavior of the nonlinearity compels us to employ truncation arguments, while the presence of both the discontinuity and the obstacle requires the usage of non-smooth calculus and convex analysis. Even more, the interplay between singularity and discontinuities casts the problem off the standard theory of non-smooth calculus for integral operators, so that a generalization of the latter is needed. More precisely, the well-known Aubin-Clarke theorem (see, e.g., \cite[Theorem 2.181]{CLM}) furnishes an expression on the Clarke generalized gradient of functionals of type
$$ \G(u):=\int_\Omega G(x,u(x)) \dx,  $$
where $G:\Omega\times\R\to\R$ is such that $G(\cdot,s)$ is Lebesgue-measurable for all $s\in\R$ and it satisfies either, for some $k\in L^{p'}(\Omega)$,
$$ |G(x,s_1)-G(x,s_2)| \leq k(x)|s_1-s_2| \quad \mbox{for a.a.} \ x\in\Omega \ \mbox{and for all} \ s_1,s_2\in\R, $$
or $G(x,\cdot)$ is locally Lipschitz continuous for a.a.\,$x\in\Omega$ and, for suitable $h\in L^{p'}(\Omega)$ and $c>0$,
$$ |t| \leq h(x)+c|s|^{p-1} \quad \mbox{for a.a.} \ x\in\Omega, \ \mbox{for all} \ s\in\R, \mbox{and for all} \ t\in\partial_s G(x,s), $$
being $\partial_s G(x,s)$ the generalized Clarke gradient of $G(x,\cdot)$. Unfortunately, dealing with singular problems with $f(s)\simeq s^{-\gamma}$ leads to $h(x)\simeq ({\rm dist}(x,\partial\Omega))^{-\gamma}$, so that in general $h\notin L^{p'}(\Omega)$ (cf. Proposition \ref{generalgrad} and Remark \ref{ggrowth} below).

Another issue, which appears also in the absence of the obstacle (i.e., $\phi\equiv 0$), is that the generalized critical points of the functional
$$ J(u):=\frac{1}{p}\int_\Omega |\nabla u|^p \dx - \int_\Omega F(u(x)) \dx, \quad u\in W^{1,p}_0(\Omega), $$
being $F:\R\to\R$ a locally Lipschitz continuous function, do not solve the Euler-Lagrange equations for $J$, but they are merely solutions of the differential inclusion $-\Delta_p u \in \partial F(u)$. If $F$ is a primitive of a function $f:\R\to\R$ fulfilling `good' growth conditions and having a Lebesgue-negligible set of discontinuity points, then the locality property of the $p$-Laplacian guarantees that $u$ is a strong solution to $-\Delta_p u = f(u)$, meaning that this differential equation is satisfied almost everywhere in the pointwise sense; cf. Definition \ref{strongsoldef} and Proposition \ref{stronglocal} below. For an account of this problem in the singular setting, we address the reader to \cite{GM,GM2}.

Motivated by these remarks, we argue as follows:
\begin{enumerate}
\item we truncate the reaction term at the level of a subsolution to $-\Delta_p u = \underline{f}(u)$, in order to avoid the possible singularity of $f$ (see Proposition \ref{subsollemma});
\item we generalize the Aubin-Clarke theorem to integral functionals whose integrand is a primitive of a non-negative function $g(x,s)\lesssim|s|^{p-1}+({\rm dist}(x,\partial\Omega))^{-\gamma}$, $(x,s)\in\Omega\times\R$ (see Proposition \ref{generalgrad});
\item we get existence of a solution to the truncated variational-hemivariational inequality via pseudomonotone operator theory, and this solution turns out to solve also \eqref{prob} (see Theorem \ref{auxexistence});
\item we prove that, far away from the contact set $\{u=\phi\}$, the solution obtained is locally $C^{1,\alpha}$-regular and strongly solves $-\Delta_p u = f(u)$ (see Definition \ref{contactdef} and Theorem \ref{locregthm});
\item we show that, if $-\Delta_p \phi_+ \leq \Phi \in L^q(\Omega)$ for some $q>N$, then a Lewy-Stampacchia type inequality (see Proposition \ref{lewystampacchia}) ensures that the solution found belongs to $C^{1,\alpha}(\overline{\Omega})$ (see Theorem \ref{globregthm}).
\end{enumerate}

We also highlight that both singular problems and obstacle problems naturally arise in physics, engineering, and chemistry. To provide some examples, we recall that singular problems serve as models for heat conduction in electrically conducting materials and stem from the study of chemical catalysts; meanwhile, obstacle problems describe equilibrium configurations under constraints, as occurs in contact problems in elasticity and in phase transition models. For the state of the art concerning singular problems, we refer the reader to the surveys \cite{GLM1,GLM2}, while for obstacle problems we address to the monograph \cite{KS} (we also refer to \cite{K,F,Lind} for some regularity results). 
Incidentally, we also report the pioneering paper \cite{C}, devoted to non-smooth calculus and its applications to PDEs. We also mention the inspiring article \cite{KP}, which investigates a variational-hemivariational inequality associated with an obstacle $p$-Laplacian problem exhibiting discontinuous reaction term.

The rest of the paper is organized as follows. Section \ref{S:obstacle} provides the statement of the main result. Then, Section \ref{S:preliminaries} is devoted to some preliminary facts about nonlinear operator theory, non-smooth differential calculus, and variational inequalities. Finally, Section \ref{S:mainres} pertains the proof of Theorem \ref{mainthm} below.

\section{The obstacle problem}
\label{S:obstacle}

\subsection*{Notation.}
Let $\M(\Omega)$ be the set of the Lebesgue-measurable functions on $\Omega$. If $f\in\M(\Omega)$, then $f_+:=\max\{f,0\}$ will denote the positive part of $f$. If $g,h\in \M(\Omega)$, we will write $g\leq h$ in $\Omega$ to signify $g(x)\leq h(x)$ for almost all $x\in\Omega$.
Anyway, to avoid technicalities, we will write `for all $x\in\Omega$' instead `for almost all $x\in\Omega$' when the meaning is clear. \\
We will abbreviate with $d$ the distance function on $\Omega$, that is, 
$$ d(x):={\rm dist}(x,\partial\Omega) \quad {\rm for \ all} \ x\in\Omega. $$
Given any distribution $\mu$ on $\Omega$, its support will be denoted by $\supp \mu$. \\
The symbol $\langle \cdot,\cdot \rangle_X$ stands for the duality pairing between the Banach space $X$ and its topological dual $X^*$; the subscript is omitted when no confusion can arise. \\
For any open set $A\subseteq\R^N$, the symbol $C^\infty_c(A)$ indicates the set of the infinitely differentiable functions having compact support in $A$. The space $W^{1,p}_0(\Omega)$ is the closure of $C^\infty_c(\Omega)$ with respect to the standard Sobolev norm. Taking into account the Poincaré inequality, hereafter $W^{1,p}_0(\Omega)$ will be endowed with the equivalent norm $\|u\|_{1,p}:=\||\nabla u|\|_p$, being $\|\cdot\|_p$ the classical Lebesgue norm. When $u\in W^{1,p}_0(\Omega)$, we will abbreviate $\||\nabla u|\|_p$ as $\|\nabla u\|_p$. The dual space of $W^{1,p}_0(\Omega)$ will be denoted by $W^{-1,p'}(\Omega)$, being $p':=\frac{p}{p-1}$. If $\eta\in W^{-1,p'}(\Omega)$ is a distribution which can be represented by a function $v\in L^q(\Omega)$, $q\geq 1$, we will simply write $v\in W^{-1,p'}(\Omega)\cap L^q(\Omega)$, identifying $\eta$ with $v$.

Let $\J:X\to\R$ be a functional defined on the Banach space $(X,\|\cdot\|)$. Given any $u,v\in X$, we indicate with $\J^0(u;v)$ the generalized directional derivative of $J$ at $u$ in the direction of $v$, that is,
$$ \J^0(u;v):=\limsup_{\substack{w\to u \\ t\to 0^+}} \frac{\J(w+tv)-\J(w)}{t}. $$
For every $u\in X$, $\partial \J(u)$ will denote the Clarke generalized gradient of $\J$ at $u$, i.e.,
$$ \partial \J(u) := \left\{ \xi\in X^*: \ \J^0(u;v) \geq \langle \xi,v \rangle \ \mbox{for all} \ v\in X \right\}. $$
If $\J$ is locally Lipschitz continuous, then $\J^0(u;v)$ is finite and $\partial\J(u)\neq\emptyset$ for all $u,v\in X$. \\ 
Let $\F:X\to\R\cup\{+\infty\}$ be a convex functional. For any $u\in \dom(\F)$ (i.e., $\F(u)<+\infty$), we define $D\F(u)$ as the subdifferential of $\F$ at $u$, that is,
$$ D\F(u) := \left\{\eta\in X^*: \ \F(v)-\F(u) \geq \langle \eta,v-u \rangle \ \mbox{for all} \ v\in X \right\}. $$
If $u\notin \dom(\F)$ then we posit $D\F(u)=\emptyset$. Incidentally, we recall that the Clarke generalized gradient and the subdifferential agree for convex, continuous functionals (see \cite[Example 2.169]{CLM}). \\
In the sequel, $\K:W^{1,p}_0(\Omega)\to[0,+\infty]$ will denote the indicator function of the set $K$ defined in \eqref{Kdef}, i.e.,
\begin{equation}
\label{Kcaldef}
\mathcal{K}(u):=\begin{cases}
0, &\mbox{if} \ u\in K, \\
+\infty, &\mbox{if} \ u\notin K.
\end{cases}
\end{equation}

\subsection*{Main results}

Let $f:(0,+\infty)\to[0,+\infty)$ and $\phi:\Omega\to\R$ be a Borel-measurable and a Lebesgue-measurable functions, respectively. Hereafter, $\lambda_1>0$ is the first eigenvalue of $(-\Delta_p,W^{1,p}_0(\Omega))$, while $\D_f\subseteq\R^+:=(0,+\infty)$ is the set of discontinuity points of $f$.

We assume:
\begin{enumerate}[label={$({\rm H}_f){\rm (\roman*)}$},ref={$({\rm H}_f){\rm (\roman*)}$},itemsep=1em,topsep=1em]
\hypertarget{Hf}{}
\item \label{localbound} $f\in L^\infty_\loc(\R^+)$;
\item \label{singular}
there exists $\gamma\in(0,1)$ such that $\displaystyle{\limsup_{s\to 0^+}}\, s^\gamma f(s)< +\infty$;
\item \label{subsol}
$\displaystyle{\liminf_{s\to 0^+}}\frac{f(s)}{s^{p-1}}>\lambda_1$;
\item\label{sublinear}
$\displaystyle{\limsup_{s\to+\infty}}\frac{f(s)}{s^{p-1}}<\lambda_1$;
\item \label{zeromeasure}
$\D_f$ has null Lebesgue measure;
\item \label{fzero}
$\underline{f}(s)=0,\ s\in\R^+\implies f(s)=0$;
\end{enumerate}
\begin{enumerate}[label={$({\rm H}_\phi)$},ref={$({\rm H}_\phi)$},itemsep=1em,topsep=1em]
\item \label{Hphi} $\phi_+\in W^{1,p}_0(\Omega)$.
\end{enumerate}

Before stating the main result of the paper, we introduce some definitions (cfr. \cite[Definition 6.7, p.43]{KS} and \cite[p.219]{GT}).

\begin{defi}
\label{contactdef}
Let $u\in W^{1,p}(\Omega)$. Given any $x_0\in\Omega$, we say that $u(x_0)>0$ in the sense of $W^{1,p}$ if there exists $\rho>0$ and a non-negative $h\in C^\infty_c(B_\rho(x_0))$ such that $h(x_0)>0$ and $u-h\geq 0$ in $B_\rho(x_0)$.
\end{defi}
We will denote by $[u>0]$ the set of the points $x\in\Omega$ such that $u(x)>0$ in the sense of $W^{1,p}$. We also set $[u<0]:=[-u>0]$ and $[u=0]:=\Omega\setminus([u>0]\cup[u<0])$. Moreover, for all $u,v\in W^{1,p}(\Omega)$, we define $[u>v]:=[u-v>0]$ (analogous definitions hold for $[u<v]$ and $[u=v]$). Notice that the sets $[u>v]$ and $[u<v]$ are open, while $[u=v]$ is closed.

\begin{defi}
\label{strongsoldef}
Let $\A\subseteq\Omega$ be an open set. Suppose $f:(0,+\infty)\to\R$ to be a Lebesgue-measurable function. A function $u\in W^{1,p}(\A)$ is said to be a strong solution to
$$ -\Delta_p u = f(u) \quad \mbox{in} \ \A $$
if $|\nabla u|^{p-2}\nabla u\in W^{1,1}_\loc(\A)$, $u>0$ a.e.\,in $\A$, and
$$ -\Delta_p u(x)=f(u(x)) \quad \mbox{a.e. in} \ \A. $$
\end{defi}

\begin{thm}
\label{mainthm}
Let \hyperlink{Hf}{$({\rm H}_f)$}--\ref{Hphi} be satisfied. Then problem \eqref{prob} admits a positive (a.e.\,in $\Omega$) solution $u\in K$, being $K$ as in \eqref{Kdef}. Moreover, setting $\A:=[u>\phi]$, the solution $u$ fulfills the following properties:
\begin{itemize}
\item $u\in C^{1,\alpha}_\loc(\A)$;
\item $|\nabla u|^{p-2}\nabla u\in W^{1,2}_\loc(\A)$;
\item $u$ strongly solves $-\Delta_p u = f(u)$ in $\A$.
\end{itemize}
If, in addition, $-\Delta_p \phi_+ \leq \Phi$ in $\Omega$ for some $\Phi\in L^q(\Omega)$, $q>N$, then $u\in C^{1,\alpha}(\overline{\Omega})$.
\end{thm}

It is worth noticing that hypothesis \hyperlink{Hf}{$({\rm H}_f)$} consists of the assumptions made in \cite{GM}, where the existence of regular solutions to \eqref{prob} in the absence of the obstacle (i.e., $\phi\equiv 0$) is investigated. Accordingly, from the existence point of view, Theorem \ref{mainthm} generalizes \cite[Theorem 1.3]{GM}. However, unlike \cite{GM}, where solutions belong to $C^{1,\alpha}(\overline{\Omega})$ and strongly solve $-\Delta_p u = f(u)$ in $\Omega$, solutions to \eqref{prob} will be of class $C^{1,\alpha}$ and strongly solve $-\Delta_p u = f(u)$ only far away from the contact set $\{u=\phi\}$. This is a classical phenomenon, due to the presence of the obstacle: see, e.g., \cite[p.105]{KS}. On the other hand, under the differentiability hypothesis $-\Delta_p \phi_+\in L^q(\Omega)$, $q>N$ (or the more general assumption stated in Theorem \ref{mainthm}), we will be able to recover the global $C^{1,\alpha}$ regularity of solutions.

\section{Preliminaries}\label{S:preliminaries}
In this section we collect some basic results concerning non-smooth differential calculus, convex analysis, and theory of variational-hemivariational inequalities.

Firstly, we introduce some monotonicity and coercivity properties for nonlinear operators. The following definitions are patterned after \cite[Definitions 2.118, 2.120, 2.121]{CLM} (cf. also \cite[Propositions 2.122-2.123]{CLM}).
\begin{defi}
Let $(X,\|\cdot\|)$ be a reflexive Banach space.
\begin{itemize}
\item An operator $A:X\to X^*$ is said to be {\it monotone} if
$$ \langle A(u)-A(v),u-v \rangle \geq 0 \quad \mbox{for all} \ u,v\in X. $$
\item An operator $A:X\to X^*$ is said to be {\it strictly monotone} if
$$ \langle A(u)-A(v),u-v \rangle > 0 \quad \mbox{for all} \ u,v\in X, \ u\neq v. $$
\item An operator $A:X\to X^*$ is said to be {\it pseudomonotone} if, for any $u_n\rightharpoonup u$ in $X$ and such that
$$ \limsup_{n\to\infty} \langle A(u_n),u_n-u \rangle \leq 0,$$
one has
$$ \langle Au,u-v \rangle \leq \liminf_{n\to\infty} \langle A(u_n),u_n-v \rangle \quad \mbox{for all} \ v\in X. $$
\item A multi-valued bounded operator $A:X\to 2^{X^*}$ is said to be (generalized) {\it pseudomonotone} if $A(u)$ is non-empty, closed, and convex for all $u\in X$ and, given any $u_n\rightharpoonup u$ in $X$ and $u_n^*\rightharpoonup u^*$ in $X^*$ fulfilling $u_n^*\in A(u_n)$ for all $n\in\N$ and
$$ \limsup_{n\to\infty} \langle u_n^*,u_n-u \rangle \leq 0,$$
one has $u^*\in A(u)$ and $\langle u_n^*,u_n \rangle \to \langle u^*,u \rangle$ as $n\to\infty$.
\item A multi-valued operator $A:X\to 2^{X^*}$ is said to be {\it maximal monotone} if the graph $\Gr(A)$ of $A$ is characterized as follows: $(u,u^*)\in\Gr(A)$ if and only if
$$ \langle u^*-v^*,u-v \rangle \geq 0 \quad \mbox{for all} \ (v,v^*)\in\Gr(A). $$
\item Given any $u_0\in X$, a multi-valued operator $A:X\to 2^{X^*}$ is said to be $u_0$-coercive if there exists $c:\R^+\to\R$ (possibly depending on $u_0$) such that
\begin{enumerate}[label={{\rm (\roman*)}}]
\item $c(r)\to+\infty$ as $r\to+\infty$;
\item $\langle u^*,u-u_0\rangle \geq c(\|u\|)\|u\|$ for all $u\in X\setminus\{0\}$ and $u^*\in A(u)$.
\end{enumerate}
\end{itemize}
\end{defi}

Now we state a surjectivity theorem for maximal monotone operators, which is a special case of \cite[Theorem 2.127]{CLM}.

\begin{thm}
\label{surjective}
Let $(X,\|\cdot\|)$ be a reflexive Banach space, $A:X\to 2^{X^*}$ be a bounded pseudomonotone operator, and $B:X\to 2^{X^*}$ be a maximal monotone operator. Suppose also that $A$ is $u_0$-coercive for some $u_0\in\dom(B)$. Then $A+B$ is surjective, i.e., for any $u^*\in X^*$ there exists $u\in X$ such that $u^*\in (A+B)(u)$.
\end{thm}

The following result, contained in \cite[Lemma 2.111]{CLM}, concerns some basic properties of the $p$-Laplacian operator.

\begin{prop}
\label{plaplacian}
The operator $-\Delta_p: W^{1,p}_0(\Omega)\to W^{-1,p'}(\Omega)$ is bounded, continuous, and strictly monotone. In particular, it is pseudomonotone.
\end{prop}

In order to handle the singularity of $f$, we will truncate it at the level of a small, regular subsolution $\underline{u}\in W^{1,p}_0(\Omega)$ to $-\Delta_p u = \underline{f}(u)$ in $\Omega$. More precisely, we state the following result, which is patterned after \cite[Lemma 3.1]{GM}.
\begin{prop}
\label{subsollemma}
Let \ref{subsol} be satisfied. Then there exists a positive $\underline{u}\in C^{1,\alpha}_0(\overline{\Omega})$, $\alpha\in(0,1)$, such that
\begin{equation}
\label{subsolprop}
\int_\Omega |\nabla \underline{u}|^{p-2}\nabla \underline{u}\nabla \varphi \dx \leq \int_\Omega \underline{f}(\underline{u})\varphi \dx \quad \mbox{for all non-negative} \ \varphi\in W^{1,p}_0(\Omega).
\end{equation}
In particular, there exists $\sigma>0$ such that
\begin{equation}
\label{hopf}
\underline{u}(x)\geq \sigma d(x) \quad \mbox{for all} \ x\in\Omega.
\end{equation}
\end{prop}
\begin{proof}
Owing to \ref{subsol}, there exists $\delta>0$ such that
\begin{equation*}
f(s) > \lambda_1 s^{p-1} \quad \mbox{for all} \ s\in(0,\delta).
\end{equation*}
In particular, by \eqref{subsupf} we get
\begin{equation}
\label{lowerfbound}
\underline{f}(s) \geq \lambda_1 s^{p-1} \quad \mbox{for all} \ s\in(0,\delta).
\end{equation}
Now take any $\underline{u}\in C^{1,\alpha}_0(\overline{\Omega})$ first eigenfunction of $(-\Delta_p,W^{1,p}_0(\Omega))$ such that $\|\underline{u}\|_\infty<\delta$. Then \eqref{lowerfbound} implies 
$$ -\Delta_p \underline{u} = \lambda_1 \underline{u}^{p-1} \leq \underline{f}(\underline{u}) \quad \mbox{in} \ \Omega, $$
so that \eqref{subsolprop} is ensured. Finally, \eqref{hopf} is an immediate consequence of \cite[Proposition 2.3]{GM}.
\end{proof}

According to Proposition \ref{subsollemma}, we can introduce the truncated reaction term $g:\Omega\times\R\to[0,+\infty)$ defined as
\begin{equation}
\label{gdef}
g(x,s) := f(\max\{\underline{u}(x),s\}) \quad \mbox{for all} \ (x,s)\in\Omega\times\R.
\end{equation}
Then we consider the function $G:\Omega\times\R\to\R$ defined as
\begin{equation}
\label{Gdef}
G(x,s) := \int_0^s g(x,t) \dt \quad \mbox{for all} \ (x,s)\in\Omega\times\R,
\end{equation}
and the associated functional $\G:W^{1,p}_0(\Omega)\to\R$, i.e.,
\begin{equation}
\label{Gcaldef}
\G(u) := \int_\Omega G(x,u(x)) \dx \quad \mbox{for all} \ u\in W^{1,p}_0(\Omega).
\end{equation}

Due to the regularity of $\underline{u}$ and the growth of $f$ (see Remark \ref{ggrowth} below), the functionals $G$ and $\G$ introduced in \eqref{Gdef} and \eqref{Gcaldef}, respectively, fulfill the properties summarized in the next proposition, which aims to generalize the Aubin-Clarke theorem in the singular context. Before stating the result, we introduce the following notation:
\begin{equation*}
\begin{aligned}
g_x(s):=g(x,s) \quad &\mbox{and} \quad G_x(s):=G(x,s) \quad &&\mbox{for all} \ (x,s)\in\Omega\times\R, \\
\underline{g_x}(s):=\lim_{\delta\to 0^+}\essinf_{|t-s|<\delta} g_x(t)\quad&\mbox{and}\quad
\overline{g_x}(s):=\lim_{\delta\to 0^+}\esssup_{|t-s|<\delta} g_x(t)\quad &&\mbox{for all} \ (x,s)\in\Omega\times\R.
\end{aligned}
\end{equation*}

\begin{prop}
\label{generalgrad}
Let $g:\Omega\times\R\to\R$ be a Borel-measurable function such that
\begin{equation}
\label{aubinclarkecond}
|g(x,s)|\leq c_1 |s|^{p-1} + c_2 d(x)^{-\gamma} \quad \mbox{for all} \ (x,s)\in\Omega\times\R,
\end{equation}
being $0<\gamma<1<p<+\infty$. Then the functions $G$ and $\G$ introduced in \eqref{Gdef} and \eqref{Gcaldef}, respectively, are well defined. In addition, $G_x$ is locally Lipschitz continuous and
\begin{equation}
\label{Gder}
\partial G_x(s) = [\underline{g_x}(s),\overline{g_x}(s)] \quad \mbox{for all} \ (x,s)\in\Omega\times\R,
\end{equation}
as well as $\G$ is locally Lipschitz continuous and
\begin{equation}
\label{Gcalder}
\partial \G(u) \subseteq \{v\in W^{-1,p'}(\Omega)\cap L^{q'}(\Omega): \, v(x)\in \partial G_x(u(x)) \ \mbox{for a.a.} \ x\in\Omega\}
\end{equation}
for all $u\in W^{1,p}_0(\Omega)$ and $q>\max\left\{p,\frac{1}{1-\gamma}\right\}$.\\
Moreover, $\partial\G:W^{1,p}_0(\Omega)\to 2^{W^{-1,p'}(\Omega)}$ is a bounded pseudomonotone operator.
\end{prop}
\begin{proof}
According to \eqref{aubinclarkecond}, the function $g_x$ belongs to $L^\infty_\loc(\R)$, so that $G$ is well defined. Then the Hardy-Sobolev inequality (see, e.g., \cite[Proposition 2.2]{GM}), jointly with \eqref{aubinclarkecond}, ensures that $\G$ is well defined on $W^{1,p}_0(\Omega)$. Indeed, for every $u \in W^{1,p}_0(\Omega)$, one has
\begin{equation}
\label{Gcalwelldef}
\begin{aligned}
\left|\int_\Omega G(x,u(x)) \dx\right| &\leq \int_\Omega \left(\int_0^{|u(x)|}|g(x,t)|\dt\right) \dx \\ &\leq \int_\Omega \left(\int_0^{|u(x)|}(c_1 t^{p-1} + c_2 d(x)^{-\gamma})\dt\right) \dx \\
&\leq \int_\Omega \left(\frac{c_1}{p}|u(x)|^p + c_2d(x)^{-\gamma}|u(x)|\right) \dx \\ &\leq C(\|u\|_p^p+\|u\|_{1,p})<+\infty,
\end{aligned}
\end{equation}
for a suitable $C=C(N,p,\Omega,c_1,c_2)>0$.

The local Lipschitz continuity of $G_x$ and the representation formula \eqref{Gder} are immediate consequences of \cite[Example 1]{LM}, applied to $g_x\in L^\infty_\loc(\R)$ for each fixed $x\in\Omega$.

Let us prove that $\G$ is Lipschitz continuous on bounded sets, so that it is locally Lipschitz continuous. Preliminarily, observe that \eqref{aubinclarkecond} entails
\begin{equation}
\label{convexineq}
|g(x,t)| \leq c_1(|a|^{p-1}+|b|^{p-1})+c_2d(x)^{-\gamma} \quad \mbox{for all} \ (x,t)\in \Omega\times[a,b],
\end{equation}
whatever $a,b\in\R$, $a<b$. Now fix an arbitrary $M>0$ and take any $u,v\in W^{1,p}_0(\Omega)$ such that $\max\{\|u\|_{1,p},\|v\|_{1,p}\}\leq M$. According to \eqref{convexineq}, as well as H\"older's, Hardy-Sobolev's, and Poincaré's inequalities, we deduce
\begin{equation}
\label{lebourg}
\begin{aligned}
|\G(u)-\G(v)| &\leq \int_\Omega |G(x,u(x))-G(x,v(x))| \dx \leq \int_\Omega \left|\int_{v(x)}^{u(x)} |g(x,t)| \dt\right| \dx \\
&\leq \int_\Omega (c_1|u(x)|^{p-1}+c_1|v(x)|^{p-1}+c_2d(x)^{-\gamma})|u(x)-v(x)| \dx \\
&\leq c_1(\|u\|_p^{p-1}+\|v\|_p^{p-1})\|u-v\|_p+c_2\|d^{-\gamma}\|_{-1,p'}\|u-v\|_{1,p} \\
&\leq C'_M\|u-v\|_{1,p},
\end{aligned}
\end{equation}
for an opportune $C'_M=C'_M(N,p,\Omega,c_1,c_2,M)>0$.

Now the inclusion \eqref{Gcalder} follows by adapting the Aubin-Clarke theorem \cite[Theorem 2.181]{CLM} to the singular setting in the following way. Fix any $u\in W^{1,p}_0(\Omega)$ and let $\{w_n\}\subseteq W^{1,p}_0(\Omega)$ be such that $w_n\to u$ in $W^{1,p}_0(\Omega)$. Then, according to \cite[Theorem 4.9]{Br}, there exists a non-negative $W\in L^p(\Omega)$ such that $|w_n|\leq W$ a.e.\,in $\Omega$. Thus, a computation similar to \eqref{lebourg} reveals that, given any $z\in W^{1,p}_0(\Omega)$ and $\{t_n\}\subseteq (0,1)$ such that $t_n\to 0^+$, there exists $C''=C''(p,c_1,c_2)>0$ such that, for all $x\in\Omega$ and $n\in\N$,
\begin{equation}
\label{domconv}
\frac{|G_x(w_n(x)+t_n z(x))-G_x(w_n(x))|}{t_n} \leq C''((W(x))^{p-1}+|z(x)|^{p-1}+d(x)^{-\gamma})|z(x)|.
\end{equation}
The right-hand side of \eqref{domconv} is independent of $n\in\N$ and, through H\"older's and Hardy-Sobolev's inequalities, it is summable in $\Omega$. In particular, the function
$$ x \mapsto G_x^0(u(x);z(x)) $$
belongs to $L^1(\Omega)$ for all $u,z\in W^{1,p}_0(\Omega)$. Accordingly, fixed any $u,z\in W^{1,p}_0(\Omega)$, the Fatou lemma ensures
\begin{equation}
\label{fatou}
\begin{aligned}
\G^0(u;z) &= \limsup_{\substack{w\to u \\ t\to 0^+}} \frac{\G(w+tz)-\G(w)}{t} \\ &= \limsup_{\substack{w\to u \\ t\to 0^+}} \int_\Omega \frac{G_x(w(x)+tz(x))-G_x(w(x))}{t} \dx \\
&\leq \int_\Omega \left[\limsup_{\substack{w\to u \\ t\to 0^+}} \frac{G_x(w(x)+tz(x))-G_x(w(x))}{t}\right] \dx \\ &= \int_\Omega G_x^0(u(x);z(x)) \dx.
\end{aligned}
\end{equation}
Let us consider $\Psi:W^{1,p}_0(\Omega)\to\R$ defined as
$$ \Psi(z) := \int_\Omega G_x^0(u(x);z(x)) \dx \quad \mbox{for all} \ z\in W^{1,p}_0(\Omega). $$
Take any $\xi\in \partial \G(u)$. Then the definition of Clarke's generalized gradient and \eqref{fatou} yield
$$ \langle \xi,z \rangle \leq \G^0(u;z) \leq \Psi(z) \quad \mbox{for all} \ z\in W^{1,p}_0(\Omega). $$
Since $\Psi(0)=0$ and $\Psi$ is a continuous, convex functional, then $\xi\in D\Psi(0)=\partial\Psi(0)$.

To conclude the proof of \eqref{Gcalder}, it remains to show that, given any $q>\max\left\{p,\frac{1}{1-\gamma}\right\}$, we have
\begin{equation}
\label{Psider}
\partial\Psi(0)\subseteq\{v\in W^{-1,p'}(\Omega)\cap L^{q'}(\Omega): \, v(x)\in \partial G_x(u(x)) \ \mbox{for a.a.} \ x\in\Omega\}.
\end{equation}
To this end, let us consider the Banach space $(X,\|\cdot\|_X)$, where $X:=W^{1,p}_0(\Omega)+L^q(\Omega)$ is the sum space of $W^{1,p}_0(\Omega)$ and $L^q(\Omega)$ endowed with the norm
$$ \|z\|_X:=\inf\left\{\|z'\|_{1,p}+\|z''\|_q: \ z=z'+z'', \ z'\in W^{1,p}_0(\Omega), \ z''\in L^q(\Omega)\right\} \quad \mbox{for all} \ z\in X. $$
We define
$$ \psi:\Omega\times\R\to\R, \quad \psi(x,s)=G_x^0(u(x);s) \quad \mbox{for all} \ (x,s)\in\Omega\times\R, $$
as well as $\psi_x:=\psi(x,\cdot)$ for all $x\in\Omega$ and
$$ \tilde{\Psi}:X\to\R, \quad \tilde{\Psi}(z) := \int_\Omega \psi(x,z(x)) \dx \quad \mbox{for all} \ z\in X. $$
Owing to \cite[Proposition 2.1]{GM} and the choice of $q$, we have $d^{-\gamma}\in L^{q'}(\Omega)$. Thus, $|z_1|^{p-1}+|z_2|^{p-1}+d^{-\gamma}\in X^*$ for all $z_1,z_2\in X$. A computation similar to \eqref{Gcalwelldef} and \eqref{lebourg} then ensures that $\tilde{\Psi}$ is well defined and locally Lipschitz continuous on $X$. Now take any $z\in X$ and $\xi\in \partial\tilde{\Psi}(z)$. Since $X^*\hookrightarrow L^{q'}(\Omega)$, then there exists $v\in L^{q'}(\Omega)$ such that
\begin{equation}
\label{represent}
\langle \xi,h \rangle = \int_\Omega v(x)h(x) \dx \quad \mbox{for all} \ h\in X.
\end{equation}
Moreover, due to the definition of $\partial\tilde{\Psi}$ and the Fatou lemma (arguing as in \eqref{fatou}), we get
\begin{equation}
\label{tildepsigrad}
\langle \xi,h \rangle \leq \tilde{\Psi}^0(z;h) \leq \int_\Omega \psi_x^0(z(x);h(x)) \dx \quad \mbox{for all} \ h\in X.
\end{equation}
Taking into account that $\psi(x,\cdot)$ is positively homogeneous, by \eqref{represent}--\eqref{tildepsigrad} we infer
$$ \int_\Omega \left[\psi^0_x(z(x);1)-v(x)\right]h(x)\dx \geq 0 \quad \mbox{for all non-negative} \ h\in X. $$
In particular we get $v(x) \leq \psi_x^0(z(x);1)$ for a.a.\,$x\in\Omega$, forcing
$$ v(x) \in \partial\psi_x(z(x)) \quad \mbox{for a.a.} \ x\in\Omega. $$
Since $W^{1,p}_0(\Omega)$ is densely embedded in $X$, then \cite[Corollary 1.2]{MR} ensures $\partial \Psi(z)=\partial \tilde{\Psi}(z)$ for all $z\in W^{1,p}_0(\Omega)$, in particular for $z=0$. From the convexity of $\psi(x,\cdot)$ and \cite[Remark 2.170]{CLM}, we deduce $\partial\psi_x(0)=D\psi_x(0)=\partial G_x(u(x))$ for all \,$x\in\Omega$. Hence, \eqref{Psider} holds true. Indeed, recalling that $X^*\hookrightarrow W^{-1,p'}(\Omega)$, one has
\begin{equation*}
\begin{aligned}
\partial\Psi(0)=\partial\tilde{\Psi}(0)&\subseteq X^*\cap \left\{v\in L^{q'}(\Omega): \ v(x)\in\partial\psi_x(0) \ \mbox{for a.a.} \ x\in\Omega\right\} \\
&= \left\{v\in W^{-1,p'}(\Omega) \cap L^{q'}(\Omega): \ v(x)\in\partial G_x(u(x)) \ \mbox{for a.a.} \ x\in\Omega\right\}.
\end{aligned}
\end{equation*}

Finally, let us consider the operator $\partial\G:W^{1,p}_0(\Omega)\to 2^{W^{-1,p'}(\Omega)}$. Since $\G$ is Lipschitz continuous on bounded sets, then $\partial\G$ is a bounded operator. Indeed, given any bounded open set $U\subseteq W^{1,p}_0(\Omega)$ and $u\in U$, we have
$$ \|\xi\|_{-1,p'} = \max_{\|h\|_{1,p}\leq 1} |\langle \xi,h \rangle| \leq \sup_{\|h\|_{1,p}\leq 1} \G^0(u;h) = \sup_{\|h\|_{1,p}\leq 1} \limsup_{\substack{w\to u \\ t\to 0^+}} \frac{\G(w+th)-\G(w)}{t} \leq L_U(\G) $$
for all $\xi\in\partial\G(u)$, being $L_U(\G)$ the Lipschitz constant of $\G_{\mid_U}$. The local Lipschitz continuity of $\G$ ensures that $\partial\G(u)$ is non-empty, closed, and convex for all $u\in W^{1,p}_0(\Omega)$ (see \cite[pp.65-66]{CLM}). In order to prove that $\partial\G$ is pseudomonotone, take any $\{u_n\}\subseteq W^{1,p}_0(\Omega)$, $u\in W^{1,p}_0(\Omega)$, $\{\xi_n\}\subseteq W^{-1,p'}(\Omega)$, and $\xi\in W^{-1,p'}(\Omega)$ such that $u_n\rightharpoonup u$ in $W^{1,p}_0(\Omega)$, $\xi_n\rightharpoonup \xi$ in $W^{-1,p'}(\Omega)$, and $\xi_n\in \partial\G(u_n)$ for all $n\in\N$. We want to prove that $\langle \xi_n,u_n \rangle \to \langle \xi,u \rangle$ and $\xi\in\partial\G(u)$. According to \eqref{Gcalder}, for each $n\in\N$ there exists $v_n\in L^1(\Omega)$ representing $\xi_n$, so that
\begin{equation}
\label{pseudomonotone}
\begin{aligned}
\langle \xi_n,u_n \rangle = \langle \xi_n,u_n-u \rangle + \langle \xi_n,u \rangle = \int_\Omega v_n(x)(u_n(x)-u(x)) \dx + \langle \xi_n,u \rangle \quad \mbox{for all} \ n\in\N.
\end{aligned}
\end{equation}
Let us analyze the right-hand side of \eqref{pseudomonotone}. Owing to the compactness of the embedding $W^{1,p}_0(\Omega)\hookrightarrow L^p(\Omega)$, we can assume that $u_n\to u$ in $L^p(\Omega)$. Since $v_n(x)\in \partial G_x^0(u_n(x))$ for a.a.\,$x\in\Omega$, by \eqref{Gder} and \eqref{aubinclarkecond} we infer
\begin{equation*}
\begin{aligned}
\left|\int_\Omega v_n(x)(u_n(x)-u(x)) \dx\right| &\leq \int_\Omega \left[c_1|u_n(x)|^{p-1}+c_2d(x)^{-\gamma}\right]|u_n(x)-u(x)| \dx \\
&\leq c_1 \|u_n\|_p^{p-1}\|u_n-u\|_p + \int_\Omega d(x)^{-\gamma}|u_n(x)-u(x)| \dx.
\end{aligned}
\end{equation*}
From $u_n\rightharpoonup u$ in $W^{1,p}_0(\Omega)$ and $u_n\to u$ in $L^p(\Omega)$ we deduce $|u_n-u|\rightharpoonup 0$ (cf. \cite[Lemma 2.1]{BG}) and $\|u_n\|_p^{p-1}\|u_n-u\|_p\to 0$, respectively. Thus,
\begin{equation}
\label{pseudomonotone1}
\lim_{n\to\infty} \int_\Omega v_n(x)(u_n(x)-u(x)) \dx = 0.
\end{equation}
Exploiting the fact that $\xi_n\rightharpoonup \xi$ in $W^{-1,p'}(\Omega)$, we get
\begin{equation}
\label{pseudomonotone2}
\lim_{n\to\infty} \langle \xi_n,u \rangle = \langle \xi,u \rangle.
\end{equation}
Putting \eqref{pseudomonotone1}--\eqref{pseudomonotone2} into \eqref{pseudomonotone} entails
$$ \lim_{n\to\infty} \langle \xi_n,u_n \rangle = \langle \xi,u \rangle. $$

The proof of the proposition is concluded once we prove that $\xi\in\partial\G(u)$. Passing to a subsequence if necessary, we can suppose that $u_n\to u$ in $L^p(\Omega)$ and a.e.\,in $\Omega$. Now take any $h\in W^{1,p}_0(\Omega)$. Then, since $\xi_n\rightharpoonup \xi$ in $W^{-1,p'}(\Omega)$ and $\xi_n\in\partial\G(u_n)$ for all $n\in\N$,
\begin{equation}
\label{closedgraph1}
\begin{aligned}
\langle \xi,h \rangle &= \lim_{n\to\infty} \langle \xi_n,h \rangle \leq \limsup_{n\to\infty} \G^0(u_n;h) \\
&= \limsup_{n\to\infty}\limsup_{\substack{z\to 0 \\ t\to 0^+}} \int_\Omega \left[\frac{1}{t}\int_{u_n(x)+z(x)}^{u_n(x)+z(x)+th(x)}g(x,\tau)\dtau\right] \dx.
\end{aligned}
\end{equation}
Thus, there exist $\{z_m\}\subseteq W^{1,p}_0(\Omega)$ and $\{t_m\}\subseteq\R^+$ such that $z_m\to 0$ in $W^{1,p}_0(\Omega)$ and a.e.\,in $\Omega$, $t_m\to 0^+$, and
\begin{equation}
\label{closedgraph2}
\begin{aligned}
&\limsup_{n\to\infty}\limsup_{\substack{z\to 0 \\ t\to 0^+}} \int_\Omega \left[\frac{1}{t}\int_{u_n(x)+z(x)}^{u_n(x)+z(x)+th(x)}g(x,\tau)\dtau\right] \dx \\
&= \lim_{n\to\infty}\lim_{m\to\infty} \int_\Omega \left[\frac{1}{t_m}\int_{u_n(x)+z_m(x)}^{u_n(x)+z_m(x)+t_m h(x)}g(x,\tau)\dtau\right] \dx.
\end{aligned}
\end{equation}
Let us consider the double sequence $\{a_{n,m}\}_{n,m}$, $n,m\in\N$, defined by
$$ a_{n,m} := \int_\Omega \left[\frac{1}{t_m}\int_{u_n(x)+z_m(x)}^{u_n(x)+z_m(x)+t_m h(x)}g(x,\tau)\dtau\right] \dx \quad \mbox{for all} \ n,m\in\N. $$
Reasoning as in \eqref{domconv}, the dominated convergence theorem ensures that, for a.a.\,$x\in\Omega$,
$$ \lim_{n\to\infty} \int_{u_n(x)+z_m(x)}^{u_n(x)+z_m(x)+t_m h(x)}g(x,\tau)\dtau = \int_{u(x)+z_m(x)}^{u(x)+z_m(x)+t_mh(x)} g(x,\tau) \dtau \quad \mbox{uniformly in} \ m\in\N. $$
Therefore we infer
$$ \lim_{n\to\infty} a_{n,m} = \int_\Omega \left[\frac{1}{t_m}\int_{u(x)+z_m(x)}^{u(x)+z_m(x)+t_m h(x)}g(x,\tau)\dtau\right] \dx \quad \mbox{uniformly in} \ m\in\N, $$
which allows to interchange the limits in the right-hand side of \eqref{closedgraph2}. Accordingly, \eqref{closedgraph1}--\eqref{closedgraph2} produce
\begin{equation}
\label{closedgraph3}
\begin{aligned}
\langle \xi,h \rangle &\leq \lim_{n\to\infty}\lim_{m\to\infty} \int_\Omega \left[\frac{1}{t_m}\int_{u_n(x)+z_m(x)}^{u_n(x)+z_m(x)+t_m h(x)}g(x,\tau)\dtau\right] \dx \\
&= \lim_{m\to\infty}\lim_{n\to\infty} \int_\Omega \left[\frac{1}{t_m}\int_{u_n(x)+z_m(x)}^{u_n(x)+z_m(x)+t_m h(x)}g(x,\tau)\dtau\right] \dx \\
&= \lim_{m\to\infty} \int_\Omega \left[\frac{1}{t_m}\int_{u(x)+z_m(x)}^{u(x)+z_m(x)+t_m h(x)}g(x,\tau)\dtau\right] \dx \\
&\leq \limsup_{\substack{z\to 0\\ t\to 0^+}} \int_\Omega \left[\frac{1}{t} \int_{u(x)+z(x)}^{u(x)+z(x)+t h(x)}g(x,\tau)\dtau \right] \dx = \G^0(u;h).
\end{aligned}
\end{equation}
Since \eqref{closedgraph3} holds for all $h\in W^{1,p}_0(\Omega)$, then $\xi\in\partial\G(u)$.
\end{proof}

\begin{rmk}
\label{ggrowth}
The function $g$ defined in \eqref{gdef} fulfills \eqref{aubinclarkecond}. Indeed, since Proposition \ref{subsollemma} ensures $\sigma d\leq \underline{u}\leq \delta$ in $\Omega$ for some $\sigma,\delta>0$, then \ref{singular} and \ref{sublinear} imply
\begin{equation*}
\begin{aligned}
0\leq g(x,s) &= f(\max\{\underline{u}(x),s\}) \leq c_1 |s|^{p-1} + c_2 \underline{u}(x)^{-\gamma} + c_3 \\
&\leq C_1 |s|^{p-1} + (C_2\sigma^{-\gamma} + C_3\diam(\Omega)^\gamma)d(x)^{-\gamma} \quad \mbox{for all} \ (x,s)\in\Omega\times\R,
\end{aligned}
\end{equation*}
being $C_1\in(0,\lambda_1)$ and $C_2,C_3>0$ opportune. A similar computation ensures the more general estimate
\begin{equation}
\label{overfgrowth}
0\leq \overline{g_x}(s) \leq \overline{f}(\max\{\underline{u}(x),s\}) \leq c_1 |s|^{p-1} + c_2 d(x)^{-\gamma} \quad \mbox{for all} \ (x,s)\in\Omega\times\R,
\end{equation}
where $c_1\in(0,\lambda_1)$ and $c_2>0$ are suitable constants.
\end{rmk}

\begin{prop}
\label{subdiff}
Let $(X,\|\cdot\|)$ be a Banach space and $K\subseteq X$ be a non-empty, closed, convex set. Denote by $\K$ the indicator function of $K$. Then $\K$ is a proper, convex, lower semi-continuous functional on $X$, $\dom(D\K)=\dom(\K)=K$, and $D\K:X\to 2^{X^*}$ is a maximal monotone operator.
\end{prop}
\begin{proof}
See \cite[Chapter II, Example 3 and Theorem 2.1]{B}.
\end{proof}

In the sequel we will also need the following refined version of the strong locality property for the $p$-Laplacian (see \cite[Proposition 2.4]{GM}).


\begin{prop}
\label{stronglocal}
Let $D\subseteq\R$ be a set having null Lebesgue measure. Suppose that, for some $v\in L^2_\loc(\Omega)$, the function $u\in W^{1,p}_\loc(\Omega)$ is a distributional solution to $\Delta_p u = v$ in $\Omega$. Then $|\nabla u|^{p-2}\nabla u\in W^{1,2}_\loc(\Omega)$ and
$$ \Delta_p u(x) = 0 \quad \mbox{for a.e.} \ x\in u^{-1}(D). $$
\end{prop}

We conclude this section by proving two results which link together differential equations and variational inequalities. To this purpose, we first establish a generalization of \cite[Theorem 6.9]{KS}, and then we infer a Lewy-Stampacchia type result, based on \cite[Theorem 2.4]{GM}.
\begin{prop}
\label{kinderstamp}
Let $\eta\in W^{-1,p'}(\Omega)$, $\phi,K$ be as in \eqref{Kdef}, and $u\in K$ be a solution to the variational inequality
\begin{equation}
\label{varineq}
\int_\Omega |\nabla u|^{p-2}\nabla u \nabla(v-u) \dx \geq \langle \eta,v-u\rangle \quad \mbox{for all} \ v\in K.
\end{equation}
Then there exists a non-negative $\mu\in W^{-1,p'}(\Omega)$ such that
\begin{equation}
\label{eqmeasure}
-\Delta_p u = \eta+\mu \quad \mbox{in} \ \Omega
\end{equation}
and
\begin{equation}
\label{support}
\supp(\mu) \subseteq [u=\phi].
\end{equation}
\end{prop}
\begin{proof}
The proof easily follows from the argument developed in \cite[pp.43-44]{KS}, so here we only sketch it.

Given any non-negative $\varpi\in W^{1,p}_0(\Omega)$ and $\eps>0$, we can choose $v:=u+\eps\varpi\in K$ in \eqref{varineq}, so that
$$ \langle -\Delta_p u - \eta,\varpi \rangle \geq 0 \quad \mbox{for all non-negative} \ \varpi\in W^{1,p}_0(\Omega). $$
Thus, setting $\mu:=-\Delta_p u-\eta\in W^{-1,p'}(\Omega)$ produces \eqref{eqmeasure}.

If $[u>\phi]=\emptyset$, then $[u=\phi]=\Omega$ (since $u\in K$ forces $[u<\phi]=\emptyset$) and \eqref{support} is trivial. Otherwise, let $x_0\in[u>\phi]$. Then there exist $\rho>0$ and a non-negative $h\in C^\infty_c(B_{2\rho}(x_0))$ such that $h>0$ in $\overline{B}_\rho(x_0)$ and
$$ u\geq \phi+h \quad \mbox{in} \ B_{2\rho}(x_0). $$
Thus, for any $\zeta\in C^\infty_c(B_\rho(x_0))$ we can find $\eps>0$ so small that
$$ u+\eps\zeta\geq \phi+\frac{1}{2}h \quad \mbox{in} \ B_\rho(x_0). $$
Hence, $u+\eps\zeta\in K$. Choosing $v:=u+\eps\zeta$ in \eqref{varineq} and dividing by $\eps$ yield
$$ \int_{B_\rho(x_0)} |\nabla u|^{p-2}\nabla u \nabla\zeta \dx \geq \langle \eta,\zeta \rangle. $$
Replacing $\zeta$ with $-\zeta$ gives
$$ \int_{B_\rho(x_0)} |\nabla u|^{p-2}\nabla u \nabla\zeta \dx = \langle \eta,\zeta \rangle \quad \mbox{for all} \ \zeta\in C^\infty_c(B_\rho(x_0)). $$
Standard partition of unity and density arguments produce
\begin{equation*}
-\Delta_p u = \eta \quad \mbox{in} \ [u>\phi].
\end{equation*}
Taking into account \eqref{eqmeasure} and $u\in K$ (whence $[u<\phi]=\emptyset$) yields \eqref{support}.
\end{proof}

\begin{prop}
\label{lewystampacchia}
Let $\eta\in W^{-1,p'}(\Omega)$, $\phi,K$ be as in \eqref{Kdef}, and $u\in K$ be a solution to the variational inequality \eqref{varineq}. Then
\begin{equation}
\label{lewy}
\eta \leq -\Delta_p u \leq \max\{-\Delta_p \phi_+, \eta\} \quad \mbox{in} \ W^{-1,p'}(\Omega),
\end{equation}
where the maximum of the distributions $-\Delta_p \phi_+$ and $\eta$ is given by the Riesz-Kantorovich formula
$$ \max\{-\Delta_p \phi_+, \eta\}(v) := \sup_{\substack{0\leq z\leq v\\ z\in W^{1,p}_0(\Omega)}} \left(\langle-\Delta_p \phi_+,z\rangle + \langle\eta,v-z\rangle\right) \quad \mbox{for all non-negative} \ v\in W^{1,p}_0(\Omega). $$
\end{prop}
\begin{proof}
The first inequality in \eqref{lewy} stems directly from Proposition \ref{kinderstamp}, since the distribution $\mu$ in \eqref{eqmeasure} is non-negative. On the other hand, the second inequality in \eqref{lewy} is a consequence of \cite[Theorem 2.4]{GigMos}, applied to the convex functional $E(v):=\frac{1}{p}\|\nabla v\|_p^p-\langle\eta,v\rangle$, $v\in W^{1,p}_0(\Omega)$, and the obstacle $\varphi:=\phi$, taking into account that \eqref{varineq} implies that $u$ minimizes $E$ over $K$, owing to \cite[Theorem 2.3]{GigMos}.
\end{proof}

\section{Proof of the main result}
\label{S:mainres}

Let us consider the following auxiliary problem:
\begin{equation}
\label{prob2}
\tag{Q}
\int_\Omega |\nabla u|^{p-2}\nabla u \nabla(v-u) \dx \geq \langle \eta,v-u\rangle \quad \mbox{for all} \ v\in K,
\end{equation}
where $K$ is defined in \eqref{Kdef} and $\eta\in\partial\G(u)$, being $\G$ defined in \eqref{Gcaldef}.
The problem \eqref{prob2} can be restated as a variational-hemivariational inequality. Therefore, we look for solutions $u\in K$ to
\begin{equation}
\label{varhemivar}
\int_\Omega |\nabla u|^{p-2}\nabla u\nabla(v-u) \dx + \K(v)-\K(u) \geq \langle \eta,v-u\rangle \quad \mbox{for all} \ v\in W^{1,p}_0(\Omega),
\end{equation}
where $\K$ is defined in \eqref{Kcaldef}. We can prove the following existence result for \eqref{varhemivar}.

\begin{thm}
\label{auxexistence}
Let \ref{localbound}--\ref{sublinear} and \ref{Hphi} be satisfied. Then \eqref{varhemivar} admits a solution $u\in K$ fulfilling $u\geq \underline{u}$ in $\Omega$. In particular, $u$ solves \eqref{prob}.
\end{thm}
\begin{proof}
This proof is inspired by \cite[Theorem 7.9]{CLM}.

Preliminarily, notice that \ref{Hphi} implies $\phi_+\in K$, so that $K\neq\emptyset$. Now we want to apply Theorem \ref{surjective} with $X:=W^{1,p}_0(\Omega)$, $A:=-\Delta_p-\partial\G$, $B:=D\K$, and an arbitrarily fixed $u_0\in K=\dom(D\K)$ (see Proposition \ref{subdiff}). Propositions \ref{plaplacian} and \ref{generalgrad} ensure that $-\Delta_p$ and $\partial\G$ are bounded and pseudomonotone operators, so $A=-\Delta_p-\partial\G$ enjoys the same properties. In addition, Proposition \ref{subdiff} guarantees the maximal monotonicity of the operator $B=D\K$.

It remains to prove that $A$ is $u_0$-coercive. To this aim, given any $v\in W^{1,p}_0(\Omega)\setminus\{0\}$ and $w\in\partial\G(v)$, \eqref{Gder}--\eqref{Gcalder} along with the H\"older, Hardy-Sobolev, and Poincaré inequalities entail
\begin{equation*}
\begin{aligned}
&\langle -\Delta_p v-w,v-u_0 \rangle \\
& = \langle -\Delta_p v,v\rangle - \langle -\Delta_p v,u_0 \rangle - \langle w,v-u_0\rangle \\
&\geq \|\nabla v\|_p^p - \|\nabla v\|_p^{p-1}\|\nabla u_0\|_p - \int_\Omega \overline{g_x}(v(x))|v(x)-u_0(x)| \dx \\
&\geq \|\nabla v\|_p^p - \|\nabla v\|_p^{p-1}\|\nabla u_0\|_p - \int_\Omega (c_1 |v(x)|^{p-1} + c_2 d(x)^{-\gamma})(|v(x)|+|u_0(x)|) \dx \\
&\geq \|\nabla v\|_p^p - \|\nabla v\|_p^{p-1}\|\nabla u_0\|_p - c_1 (\|v\|_p^p+\|v\|_p^{p-1}\|u_0\|_p) - c_2 \|d^{-\gamma}\|_{-1,p'}(\|v\|_{1,p}+\|u_0\|_{1,p}) \\
&\geq \left(1-\frac{c_1}{\lambda_1}\right)(\|\nabla v\|_p^p - \|\nabla u_0\|_p\|\nabla v\|_p^{p-1}) - c_2 \|d^{-\gamma}\|_{-1,p'}(\|\nabla v\|_p +\|\nabla u_0\|_p),
\end{aligned}
\end{equation*}
where $c_1\in(0,\lambda_1)$ and $c_2>0$ stem from Remark \ref{ggrowth}. Hence we have
$$ \langle -\Delta_p v-w,v-u_0 \rangle \geq \left[ \left(1-\frac{c_1}{\lambda_1}\right)\|\nabla v\|_p^{p-1}-C_{u_0}\left(\|\nabla v\|_p^{p-2}+\|\nabla v\|_p^{-1}+1\right) \right] \|\nabla v\|_p, $$
for a suitable $C_{u_0}>0$. Thus, choosing $c(r):=\left(1-\frac{c_1}{\lambda_1}\right)r^{p-1}-C_{u_0}\left(r^{p-2}+r^{-1}+1\right)$ in the definition of $u_0$-coercivity, it turns out that $A$ is $u_0$-coercive.

According to Theorem \ref{surjective}, we conclude that there exist $u\in W^{1,p}_0(\Omega)$, $\xi\in D\K(u)$, and $\eta\in\partial\G(u)$ such that
$$ -\Delta_p u + \xi - \eta = 0 \quad \mbox{in} \ W^{-1,p'}(\Omega). $$
By the definition of $D\K(u)$ we get
\begin{equation}
\label{convindicator}
\K(v)-\K(u) \geq \langle \xi,v-u \rangle \quad \mbox{for all} \ v\in W^{1,p}_0(\Omega), 
\end{equation}
whence \eqref{varhemivar}. Moreover, taking any $v\in K$ (so that $\K(v)=0$), \eqref{convindicator} yields
$$ \K(u) \leq \K(v) - \langle \xi,v-u\rangle \leq \|\xi\|_{-1,p'}\|v-u\|_{1,p} < +\infty, $$
which implies $u\in K$.

Finally, we prove that $u\geq\underline{u}$ in $\Omega$. According to \eqref{subsolprop}, we have
\begin{equation}
\label{subsoltest}
\int_\Omega |\nabla\underline{u}|^{p-2}\nabla\underline{u}\nabla(\underline{u}-u)_+ \dx \leq \int_\Omega \underline{f}(\underline{u})(\underline{u}-u)_+ \dx.
\end{equation}
Moreover, since $u\in K$, from \eqref{varhemivar} we deduce
\begin{equation}
\label{varhemivar2}
\int_\Omega |\nabla u|^{p-2}\nabla u\nabla(v-u) \dx \geq \int_\Omega \underline{g_x}(u)(v-u) \dx \quad \mbox{for all} \ v\in W^{1,p}_0(\Omega), \ v\geq u,
\end{equation}
according to Proposition \ref{generalgrad} (notice that $\eta(x)\in\partial G_x(u(x))=[\underline{g_x}(u(x)),\overline{g_x}(u(x))]$ for a.a.\,in $\Omega$). Thus, choosing $v:=u+(\underline{u}-u)_+$ in \eqref{varhemivar2}, we get
\begin{equation}
\label{compprinc}
\int_\Omega |\nabla u|^{p-2}\nabla u\nabla(\underline{u}-u)_+ \dx \geq \int_\Omega \underline{g_x}(u)(\underline{u}-u)_+ \dx.
\end{equation}
Subtracting \eqref{compprinc} from \eqref{subsoltest} we obtain
\begin{equation*}
\int_{\{\underline{u}\geq u\}} (|\nabla\underline{u}|^{p-2}\nabla\underline{u}-|\nabla u|^{p-2}\nabla u) \nabla(\underline{u}-u) \dx \leq \int_{\{\underline{u}\geq u\}} (\underline{f}(\underline{u})-\underline{g_x}(u))(\underline{u}-u) \dx.
\end{equation*}
Observe that, by the definition of $g$,
$$ \underline{g_x}(u) \geq \underline{f}(\max\{\underline{u},u\})=\underline{f}(\underline{u})  \quad \mbox{in} \ \{\underline{u}\geq u\}, $$
which entails
\begin{equation*}
\int_{\{\underline{u}\geq u\}} (|\nabla\underline{u}|^{p-2}\nabla\underline{u}-|\nabla u|^{p-2}\nabla u) \nabla(\underline{u}-u) \dx \leq 0.
\end{equation*}
Then the strict monotonicity of the $p$-Laplacian (see \cite[Chapter 12, (I) and (VII)]{Li}) forces $\{\underline{u}\geq u\}$ to be of null Lebesgue measure, so that $u\geq \underline{u}$ in $\Omega$. Hence, noticing that Proposition \ref{generalgrad}, along with \eqref{gdef} and \eqref{subsupf}, implies
$$ \eta(x)\in\partial G_x(u(x))=[\underline{g_x}(u(x)),\overline{g_x}(u(x))]\subseteq[\underline{f}(u(x)),\overline{f}(u(x))] \quad \mbox{for a.a.} \ x\in\Omega, $$
it turns out that $u$ also solves \eqref{prob}.
\end{proof}

\begin{thm}
\label{locregthm}
Let \hyperlink{Hf}{$({\rm H}_f)$}--\ref{Hphi} be satisfied. Suppose that $u\in K$ is a solution to \eqref{prob} fulfilling
\begin{equation}
\label{oversubsol}
u \geq \underline{u} \quad \mbox{in} \ \Omega \quad \mbox{for some positive} \ \underline{u}\in C^{1,\alpha}_0(\overline{\Omega}).
\end{equation}
Then $u\in C^{1,\alpha}_\loc([u>\phi])$, for some $\alpha\in(0,1)$, and
strongly solves
\begin{equation}
\label{farcontacteq}
-\Delta_p u = f(u) \quad \mbox{in} \ [u>\phi].
\end{equation}
\end{thm}
\begin{proof}
Observe that $\underline{u}\in C^{1,\alpha}_0(\overline{\Omega})$ and $\underline{u}>0$ in $\Omega$ imply that $\underline{u}\geq \sigma d$ in $\Omega$ for some $\sigma>0$, owing to \cite[Proposition 2.3]{GM}. Pick any open set $\A$ such that $\overline{\A}\subseteq [u>\phi]$ and notice that ${\rm dist}(\A,\partial\Omega)>0$, since $[u>\phi]\subseteq\Omega$ is open. Thus, according to Proposition \ref{kinderstamp} and \eqref{overfgrowth}, we get
\begin{equation}
\label{localsol}
-\Delta_p u = \eta \quad \mbox{in} \ \A,
\end{equation}
with $\eta$ satisfying
\begin{equation}
\label{localineq}
0\leq \underline{f}(u)\leq \eta\leq \overline{f}(u)\leq c_1 u^{p-1}+c_2 d^{-\gamma} \leq c_1 u^{p-1} + c_2({\rm dist}(\A,\partial\Omega))^{-\gamma} \quad \mbox{a.e. in} \ \A.
\end{equation}
Hence, nonlinear regularity theory \cite{DB} guarantees $u\in C^{1,\alpha}_\loc(\A)$, while \cite{CM} ensures that $u$ strongly solves \eqref{localsol}. Accordingly, Proposition \ref{stronglocal}, jointly with \eqref{localineq} and \ref{zeromeasure}, reveals that
\begin{equation*}
0\leq \underline{f}(u) \leq \eta = -\Delta_p u = 0 \quad \mbox{a.e. in} \ u^{-1}(\D_f)\cap \A,
\end{equation*}
so $f(u)=0$ by \ref{fzero}, forcing $\eta=f(u)=0$ a.e. in $u^{-1}(\D_f)\cap \A$. On the other hand, through \eqref{localineq} one has
\begin{equation*}
\eta=f(u) \quad \mbox{a.e. in} \ \A\setminus u^{-1}(\D_f).
\end{equation*}
Consequently $\eta=f(u)$ a.e.\,in $\A$. Since $\A$ was arbitrary, then $u\in C^{1,\alpha}_\loc([u>\phi])$ and it strongly solves \eqref{farcontacteq}.
\end{proof}

\begin{thm}
\label{globregthm}
Let \hyperlink{Hf}{$({\rm H}_f)$}--\ref{Hphi} be satisfied. Suppose that $u\in K$ is a solution to \eqref{prob} fulfilling \eqref{oversubsol}. Suppose also that there exist $q>N$ and $\Phi\in L^q(\Omega)$ such that $-\Delta_p \phi_+ \leq \Phi$ in $\Omega$. Then $u\in C^{1,\alpha}(\overline{\Omega})$ for some $\alpha\in(0,1)$.
\end{thm}
\begin{proof}
Replacing $\Phi$ with $\Phi_+$ if necessary, we may assume $\Phi$ to be non-negative.

Reasoning as in Theorem \ref{locregthm}, through \eqref{overfgrowth} we get
$$ 0\leq \eta \leq \overline{f}(u) \leq c_1 u^{p-1} + c_2 d^{-\gamma} \quad \mbox{in} \ \Omega. $$
Owing to Proposition \ref{lewystampacchia} we infer
$$ \eta \leq -\Delta_p u \leq \max\{-\Delta_p \phi_+,\eta\} \leq \Phi + c_1 u^{p-1} + c_2 d^{-\gamma} \quad \mbox{in} \ \Omega. $$
In particular, there exists a non-negative $\Gamma\in L^\infty(\Omega)$ such that $u$ solves
\begin{equation}
\label{Gammaeq1}
-\Delta_p u = \Gamma(x)\left[\Phi(x) + c_1 u^{p-1} + c_2 d(x)^{-\gamma}\right] \quad \mbox{in} \ \Omega.
\end{equation}
Let us consider the unique solution $U\in C^{1,\alpha}_0(\overline{\Omega})$ to
\begin{equation*}
\left\{
\begin{alignedat}{2}
-\Delta U &= c_2\Gamma(x)d(x)^{-\gamma} \quad &&\mbox{in} \ \Omega, \\
U&=0 \quad &&\mbox{on} \ \partial\Omega,
\end{alignedat}
\right.
\end{equation*}
whose existence and uniqueness are ensured by Minty-Browder's theorem (since $\Gamma d^{-\gamma}\in W^{-1,p'}(\Omega)$), while its regularity is guaranteed by \cite[Lemma 3.1]{H}. Then \eqref{Gammaeq1} can be rewritten as
\begin{equation}
\label{Gammaeq2}
-{\rm div}\left(|\nabla u|^{p-2}\nabla u-\nabla U(x)\right) = \Gamma(x)\left[\Phi(x) + c_1 u^{p-1}\right] \quad \mbox{in} \ \Omega.
\end{equation}
Thus, \cite[Theorem 3.1]{BCM} entails $u\in L^\infty(\Omega)$, so that \eqref{Gammaeq2} becomes
\begin{equation*}
-{\rm div}\left(|\nabla u|^{p-2}\nabla u-\nabla U(x)\right) = \Upsilon(x) \quad \mbox{in} \ \Omega,
\end{equation*}
for a suitable non-negative $\Upsilon\in L^q(\Omega)$. Hence, owing to \cite[Theorem 1]{L} (see also \cite[Corollary 1.4]{A}), we deduce that $u\in C^{1,\alpha}(\overline{\Omega})$.
\end{proof}
\begin{rmk}
If \ref{subsol} is replaced by the stronger assumption
$$ \liminf_{s\to 0^+} f(s)>0, $$
then \eqref{oversubsol} is always satisfied. Indeed, arguing as in \cite[Lemma 2.9]{CGL}, we can choose $\eps,\delta>0$ such that
$$ f(s)>\eps \quad \mbox{for all} \ s\in(0,\delta). $$
Thus, reasoning as in Lemma \ref{subsollemma}, there exists a non-negative $\underline{u}\in C^{1,\alpha}_0(\overline{\Omega})$ such that $\|\underline{u}\|_\infty<\min\left\{\left(\frac{\eps}{\lambda_1}\right)^{\frac{1}{p-1}},\delta\right\}$ and
$$ -\Delta_p \underline{u}=\lambda_1\underline{u}^{p-1}\leq \eps \leq \essinf_{s\in\left(0,\delta\right)} \underline{f}(s) \quad \mbox{in} \ \Omega. $$
Thus, choosing $v:=u+(\underline{u}-u)_+\in K$ in \eqref{prob} we get
\begin{equation*}
\begin{aligned}
\int_\Omega |\nabla u|^{p-2}\nabla u\nabla (\underline{u}-u)_+ \dx &\geq \langle \eta,(\underline{u}-u)_+ \rangle \geq \int_\Omega \underline{f}(u)(\underline{u}-u)_+ \dx \\
&\geq \eps \int_\Omega (\underline{u}-u)_+ \dx \geq \int_\Omega |\nabla\underline{u}|^{p-2}\nabla\underline{u}\nabla(\underline{u}-u)_+ \dx,
\end{aligned}
\end{equation*}
whence $u\geq \underline{u}$ in $\Omega$, using the monotonicity of the $p$-Laplacian as in Theorem \ref{auxexistence}.
\end{rmk}

Now we are able to prove Theorem \ref{mainthm}.

\begin{proof}[Proof of Theorem \ref{mainthm}]
Theorem \ref{auxexistence} provides a solution $u\in K$ to \eqref{prob} fulfilling \eqref{oversubsol}. Consequently, Theorems \ref{locregthm}--\ref{globregthm} hold.
\end{proof}

\section*{Acknowledgments}
\noindent
The authors warmly thank Sunra Mosconi and Nikolaos S. Papageorgiou for the fruitful discussions about some topics of the present research. \\
The authors are members of the {\em Gruppo Nazionale per l'Analisi Matematica, la Probabilit\`a e le loro Applicazioni} (GNAMPA) of the {\em Istituto Nazionale di Alta Matematica} (INdAM). \\
The second author acknowledges the support of the INdAM-GNAMPA project ``Classificazione delle soluzioni di equazioni evolutive non locali'' (CUP E53C25002010001).

\end{document}